\documentclass{amsart}
\usepackage{amssymb}
\newtheorem{theorem}{Theorem}[]
\newtheorem{lemma}[theorem]{Lemma}
\newtheorem{question}[theorem]{Question}
\newcommand{\e}{\ell^1}
\def\to{\longrightarrow}
\def\l{\lambda}
\def\L{\Lambda}
\def\d{\cdot}
\newcommand{\gd}{\delta}
\begin{document}
\baselineskip14pt
%%%%%%%%%%%%%%%%%%%%%%%%%%%%%%%%%%%%%%%%%%%%%%%%%%%%%%%%%%%%%%%%%%%%%%%%%%%%%%%%%%%%%%%%%%%%%%%%%%%%%%%%%%%%%%%%%%%%%%%%%%
\title[Pseudo-Amenability]{Pseudo-Amenability of Brandt Semigroup Algebras}
%%%%%%%%%%%%%%%%%%%%%%%%%%%%%%%%%%%%%%%%%%%%%%%%%%%%%%%%%%%%%%%%%%%%%%%%%%%%%%%%%%%%%%%%%%%%%%%%%%%%%%%%%%%%%%%%%%%%%%%%%%
\author[M.M. Sadr]{Maysam Maysami Sadr}
\email{sadr@iasbs.ac.ir}
\address{Department of  Mathematics,
Institute for Advanced Studies in Basic Sciences (IASBS),
Zanjan, Iran.}
%%%%%%%%%%%%%%%%%%%%%%%%%%%%%%%%%%%%%%%%%%%%%%%%%%%%%%%%%%%%%%%%%%%%%%%%%%%%%%%%%%%%%%%%%%%%%%%%%%%%%%%%%%%%%%%%%%%%%%%%%%
%%%%%%%%%%%%%%%%%%%%%%%%%%%%%%%%%%%%%%%%%%%%%%%%%%%%%%%%%%%%%%%%%%%%%%%%%%%%%%%%%%%%%%%%%%%%%%%%%%%%%%%%%%%%%%%%%%%%%%%%%%
\begin{abstract}
In this paper it is shown that for a Brandt semigroup $S$ over a
group $G$ with an arbitrary index set $I$, if $G$ is amenable, then the Banach
semigroup algebra $\e(S)$ is pseudo-amenable.
\end{abstract}
%%%%%%%%%%%%%%%%%%%%%%%%%%%%%%%%%%%%%%%%%%%%%%%%%%%%%%%%%%%%%%%%%%%%%%%%%%%%%%%%%%%%%%%%%%%%%%%%%%%%%%%%%%%%%%%%%%%%%%%%%%
\maketitle
%%%%%%%%%%%%%%%%%%%%%%%%%%%%%%%%%%%%%%%%%%%%%%%%%%%%%%%%%%%%%%%%%%%%%%%%%%%%%%%%%%%%%%%%%%%%%%%%%%%%%%%%%%%%%%%%%%%%%%%%%%
\section{Introduction}
The concept of amenability for Banach algebras was introduced by Johnson
in 1972 \cite{J}. Several modifications of this notion, such as approximate amenability
and pseudo-amenability, were introduced in \cite{GL} and \cite{GZ}.
In the current paper we investigate the pseudo-amenability of Brandt semigroup algebras.
It was shown in \cite{GL} and \cite{GZ} that for the group algebra $L^1(G)$,
amenability, approximate amenability and pseudo-amenability coincide and are
equivalent to the amenability of locally compact group $G$. In the semigroup case we know
that, if $S$ is a discrete semigroup, then amenability of $\e(S)$ implies that $S$ is
regular and amenable \cite{DN}. Ghahramani et al \cite{GLZ} have shown that, if $\e(S)$ is
approximately amenable, then $S$ is regular and amenable.
The present author and Pourabbas in \cite{MP} have shown that for a Brandt semigroup
$S$ over a group $G$ with an index set $I$, the following are equivalent.
\begin{enumerate}
\item [(i)]  $\e(S)$ is amenable.
\item [(ii)]  $\e(S)$ is approximately amenable.
\item [(iii)]  $I$ is finite and $G$ is amenable.
\end{enumerate}
This result corrects \cite[Theorem 1.8]{LS}.
In the present paper we show that for a Brandt semigroup $S$ over a group $G$
with an arbitrary (finite or infinite) index set $I$, amenability of $G$ implies pseudo-amenability
of $\e(S)$.

%%%%%%%%%%%%%%%%%%%%%%%%%%%%%%%%%%%%%%%%%%%%%%%%%%%%%%%%%%%%%%%%%%%%%%%%%%%%%%%%%%%%%%%%%%%%%%%%%%%%%%%%%%%%%%%%%%%%%%%%%%
\section{Preliminaries}
Throughout $\hat{\otimes}$ denotes the completed projective
tensor product. For an element $x$ of a set $X$, $\gd_x$ is its
point mass measure in $\e(X)$. Also, we frequently use the identification
$\e(X\times Y)=\e(X)\hat{\otimes}\e(Y)$ for the sets $X$ and $Y$.\\
A Banach algebra $A$ is called (approximately) amenable,
if for any {\it dual} Banach $A$-bimodule
$E$, every bounded {\it derivation} from $A$ to $E$ is
({\it approximately}) {\it inner}.
It is well known that amenability of $A$ is equivalent
to existence of a {\it bonded approximate diagonal}, that
is a bounded net $(m_i)\in A\hat{\otimes}A$ such that
for every $a\in A$, $a\d m_i-m_i\d a\to 0$ and $\pi(m_i)a\to a$,
where $\pi:A\hat{\otimes}A\to A$ is the continuous
bimodule homomorphism defined by $\pi(a\otimes b):=ab$ ($a,b\in A$),
and called the {\it diagonal map}.
The famous Johnson Theorem \cite{J}, says that, for any locally compact
group $G$, amenability of $G$ and $L^1(G)$ are equivalent.
For a modern account on amenability see \cite{R} and for approximate
amenability see the original papers \cite{GL} and \cite{GLZ}.\\
A Banach algebra $A$ is called pseudo-amenable (\cite{GZ})
if there is a net $(n_i)\in A\hat{\otimes}A$,
called an {\it approximate diagonal} for $A$,
such that $a\d n_i-n_i\d a\to 0$ and $\pi(n_i)a\to a$
for each $a\in A$.

Let $I$ be a nonempty set and let $G$ be a discrete group.
Consider the set $T:=I\times G\times I$,
add a null element $\o$ to $T$, and
define a semigroup multiplication  on $S:=T\cup\{\o\}$,
as follows. For
$i,i',j,j'\in I$ and $g,g'\in G$, let
$$(i,g,j)(i',g',j')=\begin{cases}
(i,gg',j') &\hbox{\rm\ if }j=i'\\
\o&\hbox{\rm\ if }j\neq i', \\
\end{cases}$$
also let $\o(i,g,j)=(i,g,j)\o=\o$ and
$\o\o=\o$. Then $S$ becomes a semigroup
that is called Brandt semigroup over $G$
with index $I$, and usually denoted by
$B(I,G)$. For more details see \cite{H}.\\
The Banach space $\e(T)$, with the convolution product,
$$(ab)(i,g,j)=\sum_{k\in I,h\in G}a(i,gh^{-1},k)b(k,h,j),$$
for $a,b\in\e(T), i,j\in I, g\in G$,
becomes a Banach algebra. (Note that if $G$ is the one point group,
and $I$ is finite, then $\e(T)$ is an ordinary matrix algebra.)
We have a closed relation between the Banach algebra
$\e(T)$ and the Banach semigroup algebra $\e(S)$:
\begin{lemma}\label{l1}
There exists an homeomorphic isomorphism $\e(S)\cong\e(T)\oplus\mathbb{C}$
of Banach algebras, where the multiplication of $\e(T)\oplus\mathbb{C}$ is coordinatewise.
\end{lemma}
\begin{proof}
Consider the following short exact sequence of Banach algebras and
continuous algebra homomorphisms.
$$0\to\e(T)\to\e(S)\to\mathbb{C}\to0,$$
where the second arrow $\Psi:\e(T)\to\e(S)$ is defined by
$\Psi(b)(t):=b(t)$ and $\Psi(b)(\o):=-\sum_{s\in T}b(s)$,
for $b\in\e(T)$ and $t\in T\subset S$,
and the third arrow $\Phi:\e(S)\to\mathbb{C}$ is the
integral functional, $\Phi(a):=\sum_{s\in S}a(s)$ ($a\in\e(S)$).
Now, let $\Theta:\e(S)\to\e(T)$ be the restriction
map, $\Theta(a):=a\mid_T$. Then $\Theta$ is a
continuous  algebra homomorphism and $\Theta\Psi=Id_{\e(T)}$.
Thus the exact sequence splits and we have $\e(S)\cong\e(T)\oplus\mathbb{C}$.
\end{proof}
\begin{lemma}\label{l2}
If $\e(T)$ is pseudo-amenable, then so is $\e(S)$.
\end{lemma}
\begin{proof}
Suppose that $\e(T)$ is pseudo-amenable, then
by Lemma \ref{l1} and \cite[Proposition 2.1]{GZ},
$\e(S)$ is pseudo-amenable.
\end{proof}
%%%%%%%%%%%%%%%%%%%%%%%%%%%%%%%%%%%%%%%%%%%%%%%%%%%%%%%%%%%%%%%%%%%%%%%%%%%%%%%%%%%%%%%%%%%%%%%%%%%%%%%%%%%%%%%%%%%%%%%%%%
\section{The Main Result}
Let $S,T,G$ and $I$, be as above.
We need some other notations and computations:\\
For $a\in\e(T)$ and every $u,v\in I$, let
$a_{(u,v)}$ be an element of $\e(G)$ defined by,
$a_{(u,v)}(g):=a(u,g,v)$ ($g\in G$). Note that
$\|a\|_{\e(T)}=\sum_{u,v\in I}\|a_{(u,v)}\|_{\e(G)}$.\\
For $b\in\e(G\times G)$, $c\in\e(G)$ and any $i,j,i',j'\in I$,
let $E^b_{(i,j,i',j')}$ and $H^c_{(i,j)}$ be elements of $\e(T\times T)$ and $\e(T)$
respectively, defined by,
$$E^b_{(i,j,i',j')}(u,g,v,u',g',v')=\begin{cases}
b(g,g') &\hbox{\rm\ if }u=i,v=j,u'=i',v'=j'\\
0&\hbox{\rm\ otherwise } \\
\end{cases}$$
$$H^c_{(i,j)}(u,g,v)=\begin{cases}
c(g) &\hbox{\rm\ if }u=i,v=j\\
0&\hbox{\rm\ otherwise } \\
\end{cases}$$
where $u,v,u',v'\in I$ and $g,g'\in G$. Also note that,
\begin{equation}\label{e1}
\|E^b_{(i,j,i',j')}\|_{\e(T\times T)}=\|b\|_{\e(G\times G)},\hspace{3mm}
\|H^c_{(i,j)}\|_{\e(T)}=\|c\|_{\e(G)}.
\end{equation}
For $u,v\in I$ and $g\in G$, the module action of $\e(T)$ on
$\e(T\times T)$ becomes,
\begin{equation}\label{e2}
\gd_{(u,g,v)}\d E^b_{(i,j,i',j')}=\begin{cases}
E^{\gd_g\d b}_{(u,j,i',j')} &\hbox{\rm\ if }i=v\\
0&\hbox{\rm\ if }i\neq v, \\
\end{cases}
\end{equation}
\begin{equation}\label{e3}
E^b_{(i,j,i',j')}\d \gd_{(u,g,v)}=\begin{cases}
E^{b\d\gd_g}_{(i,j,i',v)} &\hbox{\rm\ if }j'=u\\
0&\hbox{\rm\ if } j'\neq u.\\
\end{cases}
\end{equation}
For the multiplication of $\e(T)$, we have,
\begin{equation}\label{e4}
\gd_{(u,g,v)}H^c_{(i,j)}=\begin{cases}
H^{\gd_g c}_{(u,j)} &\hbox{\rm\ if }i=v\\
0&\hbox{\rm\ if }i\neq v, \\
\end{cases}\hspace{3mm}
H^c_{(i,j)}\gd_{(u,g,v)}=\begin{cases}
H^{c\gd_g}_{(i,v)} &\hbox{\rm\ if }j=u\\
0&\hbox{\rm\ if }j\neq u. \\
\end{cases}
\end{equation}
And finally, the diagonal maps  $\pi:\e(T\times T)\to\e(T)$ and
$\pi:\e(G\times G)\to\e(G)$ have the relation,
\begin{equation}\label{e5}
\pi(E^b_{(i,j,i',j')})=\begin{cases}
H^{\pi(b)}_{(i,j')} &\hbox{\rm\ if }j=i'\\
0&\hbox{\rm\ if } j\neq i'.\\
\end{cases}
\end{equation}
We are now ready to prove our main result:
\begin{theorem}
Suppose that $G$ is amenable. Then $\e(S)$ is pseudo-amenable.
\end{theorem}
\begin{proof}
Let $(m_\l)_{\l\in\L}\in\e(G\times G)$ be a bounded approximate
diagonal for the amenable Banach algebra $\e(G)$.
For any finite nonempty subset $F$ of $I$ and $\l\in\L$, let
$$W_{F,\l}:=\frac{1}{\#F}\sum_{i,j\in F}E^{m_\l}_{(i,j,j,i)},$$
where $\#F$ denotes the cardinal of $F$.
We show that the net
$(W_{F,\l})\in\e(T\times T)$,  over the  directed set
$\Gamma\times\L$ where $\Gamma$ is the directed set of finite subsets of $I$
ordered by inclusion,
is an approximate diagonal for $\e(T)$.\\
For any $u,v\in I$ and $g\in G$, by Equations (\ref{e2}) and (\ref{e3}), we have,
\begin{equation*}
\gd_{(u,g,v)}\d W_{F,\l}=\begin{cases}
\frac{1}{\#F}\sum_{j\in F}E^{\gd_g\d m_\l}_{(u,j,j,v)} &\hbox{\rm\ if }v\in F\\
0&\hbox{\rm\ if }v\notin F, \\
\end{cases}
\end{equation*}
\begin{equation*}
W_{F,\l}\d\gd_{(u,g,v)}=\begin{cases}
\frac{1}{\#F}\sum_{j\in F}E^{m_\l\d\gd_g}_{(u,j,j,v)} &\hbox{\rm\ if }u\in F\\
0&\hbox{\rm\ if }u\notin F, \\
\end{cases}
\end{equation*}
and thus,
\begin{equation*}
\gd_{(u,g,v)}\d W_{F,\l}-W_{F,\l}\d\gd_{(u,g,v)}=\begin{cases}
\frac{1}{\#F}\sum_{j\in F}E^{\gd_g\d m_\l-m_\l\d\gd_g}_{(u,j,j,v)} &\hbox{\rm\ if }u\in F,v\in F\\
\frac{1}{\#F}\sum_{j\in F}E^{\gd_g\d m_\l}_{(u,j,j,v)} &\hbox{\rm\ if }v\in F,u\notin F\\
-\frac{1}{\#F}\sum_{j\in F}E^{m_\l\d\gd_g}_{(u,j,j,v)} &\hbox{\rm\ if }u\in F,v\notin F\\
0&\hbox{\rm\ if }v\notin F,u\notin F. \\
\end{cases}
\end{equation*}
Then for $a=\sum_{u,v\in I,g\in G}a(u,g,v)\gd_{(u,g,v)}$ in $\e(T)$, we have,
\begin{equation*}
\begin{split}
a\d W_{F,\l}-W_{F,\l}\d a&=\frac{1}{\#F}\sum_{j,u,v\in F}E^{a_{(u,v)}\d m_\l-m_\l\d a_{(u,v)}}_{(u,j,j,v)}\\
&+\frac{1}{\#F}\sum_{j,v\in F,u\in I-F}E^{a_{(u,v)}\d m_\l}_{(u,j,j,v)}\\
&-\frac{1}{\#F}\sum_{j,u\in F,v\in I-F}E^{m_\l\d a_{(u,v)}}_{(u,j,j,v)},
\end{split}
\end{equation*}
and thus, by (\ref{e1}),
\begin{equation}\label{e9}
\begin{split}
\|a\d W_{F,\l}-W_{F,\l}\d a\|&\leq\sum_{u,v\in F}\|a_{(u,v)}\d m_\l-m_\l\d a_{(u,v)}\|\\
&+\sum_{v\in F,u\in I-F}\|a_{(u,v)}\d m_\l\|\\
&+\sum_{u\in F,v\in I-F}\|m_\l\d a_{(u,v)}\|.
\end{split}
\end{equation}
Now, suppose that $M>0$ is a bound for the norms of $m_\l$'s. Let $\epsilon>0$ be
arbitrary, and let $F_0$ be an element of $\Gamma$ such that,
$$\sum_{(u,v)\in J_0,g\in G}|a(u,g,v)|=\sum_{(u,v)\in J_0}\|a_{(u,v)}\|<\epsilon,$$
where $J_0=(I\times(I-F_0))\cup((I-F_0)\times I)$. And, choose a $\l_0\in\L$ such that for
every $\l\geq\l_0$,
$$\sum_{u,v\in F_0}\|a_{(u,v)}\d m_\l-m_\l\d a_{(u,v)}\|<\epsilon.$$
Now, if $(F,\l)\in\Gamma\times\L$ such that $F_0\subseteq F$, $\l\geq\l_0$, then we have,
\begin{equation*}
\begin{split}
\sum_{u,v\in F}\|a_{(u,v)}\d m_\l-m_\l\d a_{(u,v)}\|&\leq\sum_{u,v\in F_0}\|a_{(u,v)}\d m_\l-m_\l\d a_{(u,v)}\|\\
&+\sum_{(u,v)_\in J_0}\|a_{(u,v)}\d m_\l\|\\
&+\sum_{(u,v)_\in J_0}\|m_\l\d a_{(u,v)}\|\\
&<\epsilon+\epsilon M+\epsilon M,
\end{split}
\end{equation*}
and analogously, $\sum_{v\in F,u\in I-F}\|a_{(u,v)}\d m_\l\|<\epsilon M$ and
$\sum_{u\in F,v\in I-F}\|m_\l\d a_{(u,v)}\|<\epsilon M$. Thus by (\ref{e9}),
we have $\|a\d W_{F,\l}-W_{F,\l}\d a\|<\epsilon+4\epsilon M$.\\
\\Therefore, we proved that $a\d W_{F,\l}-W_{F,\l}\d a\to0$, for every $a\in\e(T)$.

Now, we prove that $\pi(W_{F,\l})a\to a$ for any $a\in\e(T)$.\\
\\By (\ref{e5}), we have
$\pi(W_{F,\l})=\frac{1}{\#F}\sum_{i,j\in F}H^{\pi( m_\l)}_{(i,i)}=\sum_{i\in F}H^{\pi( m_\l)}_{(i,i)}$.
Thus (\ref{e4}) implies that,
$$\pi(W_{F,\l})a=\sum_{i\in F,v\in I}H^{\pi( m_\l)a_{(i,v)}}_{(i,v)},$$
since $a=\sum_{u,v\in I}H^{a_{(u,v)}}_{(u,v)}$. Then we have,
\begin{equation}\label{e10}
\begin{split}
\|\pi(W_{F,\l})a-a\|&\leq\sum_{i\in F,v\in I}\|H^{\pi( m_\l)a_{(i,v)}-a_{(i,v)}}_{(i,v)}\|\\
&+\sum_{v\in I,u\in I-F}\|H^{a_{(u,v)}}_{(u,v)}\|.
\end{split}
\end{equation}
Let $\epsilon>0$ be arbitrary, and let $F_0$ and $J_0$ be as above.
Choose a $\l_1\in\L$ such that for every $\l\geq\l_1$,
$$\sum_{i,j\in F_0}\|\pi(m_\l)a_{(i,j)}-a_{(i,j)}\|<\epsilon.$$
Now, if $(F,\l)\in\Gamma\times\L$ such that $F_0\subseteq F$, $\l\geq\l_1$,
then by (\ref{e1}) we have,
\begin{equation*}
\begin{split}
\sum_{i\in F,v\in I}\|H^{\pi( m_\l)a_{(i,v)}-a_{(i,v)}}_{(i,v)}\|&\leq\sum_{i,j\in F_0}\|\pi(m_\l)a_{(i,j)}-a_{(i,j)}\|\\
&+\sum_{(u,v)\in J_0}\|\pi(m_\l)a_{(u,v)}\|+\sum_{(u,v)\in J_0}\|a_{(u,v)}\|\\
&<\epsilon+\epsilon M+\epsilon,
\end{split}
\end{equation*}
and, $\sum_{v\in I,u\in I-F}\|H^{a_{(u,v)}}_{(u,v)}\|=\sum_{v\in I,u\in I-F}\|a_{(u,v)}\|<\epsilon$.
Thus by (\ref{e10}), we have,
$$\|\pi(W_{F,\l})a-a\|<3\epsilon+\epsilon M.$$
This completes the proof.
\end{proof}
We end with a natural question:
\begin{question}
Dose pseudo-amenability of $\e(B(I,G))$ imply amenability of $G$?
\end{question}

%%%%%%%%%%%%%%%%%%%%%%%%%%%%%%%%%%%%%%%%%%%%%%%%%%%%%%%%%%%%%%%%%%%%%%%%%%%%%%%%%%%%%%%%%%%%%%%%%%%%%%%%%%%%%%%%%%%

\end{document}